\documentclass{article}

\usepackage{amsmath,amsthm,amssymb,mathtools}  
\usepackage{dsfont}				

\usepackage{enumerate}
\usepackage[latin1]{inputenc}   

\usepackage[hyperindex,breaklinks]{hyperref}	
\hypersetup{colorlinks=true, linkcolor=blue, citecolor=blue}
\usepackage{orcidlink}			

\usepackage{listings}		

\usepackage{cleveref} 		
\usepackage{bm}				

\usepackage{changepage}		

\usepackage{longtable}
\usepackage{booktabs}		

\usepackage[
disable,					
textsize=scriptsize,textwidth=4.2cm]{todonotes}	
\newcommand{\SELF}[1]{\todo[color=green!40]{#1}} 
\newcommand{\OMIT}[1]{\todo[color=gray!30]{#1}}  

\newtheorem{theorem}{Theorem} 
\newtheorem{proposition}[theorem]{Proposition}
\newtheorem*{proposition*}{Proposition}

\theoremstyle{definition}

\newtheorem*{example*}{Example}

\newtheoremstyle{named}{}{}{\itshape}{}{\bfseries}{.}{.5em}{\thmnote{#3}}
\theoremstyle{named}

\def\N			{\mathds{N}}

\def\R			{\mathds{R}}
\def\C			{\mathds{C}}
\def\F			{\mathds{F}}
\def\P			{\mathds{P}}

\def\II			{\mathcal{I}}		
\def\PP			{\mathcal{P}}
\def\ii			{\mathbf{i}}				
\def\jj			{\mathbf{j}}				
\def\kk			{\mathbf{k}}				
\def\ll			{\mathbf{l}}				
				
\def\qq			{\mathbf{q}}

\def\Gr			{\mathrm{Gr}}

\newcommand\pairs[2] 		{|#1>#2|}					

\begin{document}

\title{Enumeration of the New Plücker-like Equations}

\author{Andr\'e L. G. Mandolesi\,\orcidlink{0000-0002-5329-7034}
               \thanks{Instituto de Matemática e Estatística, Universidade Federal da Bahia, Av. Milton Santos s/n, 40170-110, Salvador - BA, Brazil. E-mail: \texttt{andre.mandolesi@ufba.br}}}
               
\date{\today}

\maketitle

\abstract{
We give a more detailed description of the new system of Plücker-like equations from \cite{Mandolesi_Contractions2}, 
discuss how it relates to the usual Plücker equations,
and correct a mistake in that article.

\vspace{.5em}

\vspace{3pt}

}

\

The usual system of Plücker equations \cite{Gallier2020,Jacobson1996} determines the decomposability or simplicity of a $p$-vector $H \in \bigwedge^p \F^n$ (for $\F = \R$ or $\C$),
and describes the Plücker embedding of the Grassmannian $\Gr(p,n)$ of $p$ subspaces of $\F^n$ as an algebraic variety in the projective space $\P(\bigwedge^p \F^n)$.
The system is notoriously redundant,
as its $\binom{n}{p-1} \cdot \binom{n}{p+1}$ quadratic equations include many trivial or repeated ones, and can be much more than the minimum of $\Delta_{n,p} = \dim \P(\bigwedge^p \F^n) - \dim \Gr(n,p) = \binom{n}{p} - 1 - p(n-p)$ equations needed to describe the embedding. 

Based on results from \cite{Eastwood2000}, we presented in \cite{Mandolesi_Contractions2} an equivalent system of Plücker-like equations,
with considerably less equations, though still more than $\Delta_{n,p}$.
Here we describe it in more detail, and prove that, contrary to what we had claimed, it has no trivial or repeated equations.
We hope it can be a step towards obtaining a minimal system of equations.

With the notation of \Cref{tab:symbols},
the usual system of Plücker equations,
for an orthonormal basis $(v_1,\ldots,v_n)$ of $\F^n$
and a $p$-vector $H = \sum_{\ii\in\II^n_p} \lambda_\ii v_\ii$
with $\lambda_\ii \in \F$,
can be written as \cite[Eq.\,(3)]{Mandolesi_Contractions2}:
\begin{equation}\label{eq:Plucker}
	\sum_{i\in\kk \backslash \jj} (-1)^{\pairs{\jj\triangle\kk}{i}} \lambda_{\jj\cup i} \lambda_{\kk \backslash i} = 0 \qquad \forall\, \jj\in\II^n_{p-1},\kk\in\II^n_{p+1}.
\end{equation}

\begin{table}[t]
	\scriptsize
	\centering
	\renewcommand{\arraystretch}{1}
	\begin{tabular}{ll}
		\toprule
		Symbol & Description 
		\\
		\cmidrule(lr){1-1} \cmidrule(lr){2-2} 
		$\II^n_p$ & $\{ (i_1,\ldots,i_p)\in\N^p : 1\leq i_1 < \cdots<i_p\leq n \}$
		\\
		$i_1\cdots i_p$ & $(i_1,\ldots,i_p)$
		\\
		$v_\ii$ & $v_{i_1}\wedge\cdots\wedge v_{i_p}$ for $v_1,\ldots,v_n \in \F^n$ and $\ii = i_1\cdots i_p$
		\\
		$|\ii|$ & Number of indices in $\ii$
		\\
		$\pairs{\ii}{\jj}$ & Number of pairs $(i,j)$ with $i\in\ii$, $j\in\jj$ and $i>j$	
		\\
		$\ii \cup \jj$, $\ii \backslash \jj$,  $\ii \triangle \jj$ & Ordered union, difference, symmetric difference of $\ii$ and $\jj$
		\\
		\bottomrule
	\end{tabular}
	\caption{Some notation from \cite{Mandolesi_Contractions2}.}
	\label{tab:symbols}
\end{table}

The new system of
\emph{Plücker-like equations}  
\cite[Thm.\,4.14]{Mandolesi_Contractions2} is similar, but now $\jj$ and $\kk$ have respectively $p-2$ and $p+2$ indices (with $2 \leq p \leq n-2$, and otherwise the system is trivial), and each quadratic term is formed moving 2 indices from $\kk$ to $\jj$:
\begin{equation}\label{eq:neoPlucker}
	\sum_{\ii \in \II_2^n, \ii \subset \kk \backslash \jj} (-1)^{\pairs{\jj\triangle\kk}{\ii}} \lambda_{\jj\cup \ii} \lambda_{\kk \backslash \ii} = 0
	 \qquad \forall\, \jj\in\II^n_{p-2},\kk\in\II^n_{p+2}.
\end{equation}
Its $\binom{n}{p-2} \cdot \binom{n}{p+2}$ equations are 
$\frac{(p + 2) (n - p + 2)}{(p - 1) (n - p - 1)}$
times less than in \eqref{eq:Plucker}.

The code in \Cref{code} (see \Cref{sc:Code and tables}), in the Python%
\footnote{Python Software Foundation, \url{https://www.python.org/}}
programming language,
can be used to generate both systems.
A more complete program, 
and tables with the systems for $4 \leq n \leq 9$, are provided as ancillary files for this article.

A first relation between the two systems is given by:

\begin{proposition}
	Each Plücker-like equation is a sum/subtraction of usual Plücker equations.
\end{proposition}
\begin{proof}
	In \eqref{eq:neoPlucker}, summing over $\ii = (i,i')$ with $i \neq i'$ instead of $i < i'$ (but keeping $\jj\cup \ii$ ordered) we just obtain the same terms twice.
	So the equation can be written as follows, where $\PP(\jj,\kk)$ is the expression at the left-hand side of \eqref{eq:Plucker},
	$\jj_{i} = \jj\cup i \in\II^n_{p-1}$ and $\kk_{i} = \kk\backslash i \in\II^n_{p+1}$:
	\begin{align*}
		\sum_{i\in \kk\backslash\jj}
		\Big(
		(-1)^{\pairs{\jj\triangle\kk}{i}} \cdot
		\sum_{i'\in\kk_{i} \backslash \jj_{i}} (-1)^{\pairs{\jj_{i}\triangle\kk_{i}}{i'}} \lambda_{\jj_{i}\cup i'} \lambda_{\kk_{i} \backslash i'} 
		\Big) &=\\
		\sum_{i\in \kk\backslash\jj}
		(-1)^{\pairs{\jj\triangle\kk}{i}} \cdot
		\PP(\jj_i,\kk_i) &= 0. \qedhere
	\end{align*}
\end{proof}

This proves \eqref{eq:Plucker} $\Rightarrow$ \eqref{eq:neoPlucker} directly.
The converse holds (as both systems are equivalent to $H$ being simple), but seems less immediate.
As we will see, in some cases a sum/subtraction of equations of \eqref{eq:neoPlucker} gives one of \eqref{eq:Plucker}, 
but it is unclear whether this is generally true.
As each Plücker-like equation combines $|\kk\backslash\jj| \geq 4$ Plücker equations,
it tends to have more terms, but many cancel out and at times a single Plücker equation is left.

\begin{example*}
For $(n,p) = (6,3)$, \Cref{tb:old} (in \Cref{sc:Code and tables}) has some of the 225 equations of \eqref{eq:Plucker}.
Removing trivial or repeated ones (e.g., equations \# 7 and 19)
leaves the 45 equations  of \Cref{tb:old_reduced}, still more than $\Delta_{6,3} = 10$.
\Cref{tb:new} has all $36$ equations of \eqref{eq:neoPlucker}.
Note, for example, that its equation \# 6 results from the sum of equations \# 15, 29 and 57, minus 43 and 71, from \Cref{tb:old}. 
Also, there are no trivial or repeated equations, and it has:
\begin{itemize}
	\item the same 30 equations of 3 terms as in \Cref{tb:old_reduced}, which however appear 4 times each in \eqref{eq:Plucker}, and only once in \eqref{eq:neoPlucker};
	
	\item 6 equations with the same 10 terms, differing by some signs.
	Adding or subtracting each pair we obtain one of $\binom{6}{2} = 15$ usual Plücker equations of 4 terms:
	e.g., adding equations \# 6 and 11 of \Cref{tb:new} we find \# 12 of \Cref{tb:old_reduced}.
\end{itemize}
This reveals a mistake in \cite{Mandolesi_Contractions2}: 
without realizing that some signs would differ,
we had claimed that \eqref{eq:neoPlucker} would give 6 equivalent equations when $|\jj \cap \kk| = p-3$, 
and that the last equation%
	\footnote{It also has a misprint: the first term should be $\lambda_{123}\lambda_{456}$, not $\lambda_{123}\lambda_{146}$.}
of \cite[Ex.\,4.18]{Mandolesi_Contractions2} resulted from $(\jj,\kk) = (1,23456)$, $(2,13456)$, $\ldots$ , or $(6,12345)$, which is false. 
\end{example*}

The next Proposition enumerates the Plücker-like equations in terms of the number of distinct terms, 
proves that there are no trivial or repeated ones,
and generalizes the relations with Plücker equations seen in the example.
It is unknown whether similar relations hold in case (\ref{it:set3}).

\begin{proposition}
	For $2 \leq p \leq n-2$, the Plücker-like equations \eqref{eq:neoPlucker} are all distinct and non-trivial. The system consists of
		\footnote{Recall that the multinomial $\binom{n}{n_1,\ldots,n_k} = \frac{n!}{n_1! \cdots  n_k!}$ is the number of ways to partition $n = n_1+\cdots+n_k$ elements into $k$ subsets of $n_1,\ldots,n_k$ elements.}:
	\begin{enumerate}[(i)]
		\item\label{it:set1}
		$\binom{n}{4, p-2, n-p-2}$ usual Plücker equations of 3 terms.
		
		\item\label{it:set2}
		$\binom{n}{1,5, p-3, n-p-3}$ equations of 10 terms, if $3 \leq p \leq n-3$.
		They are partitioned into $\binom{n}{6, p-3, n-p-3}$ sets of 6 equations $\{E_1,\ldots,E_6\}$ with the same terms, such that, for $1 \leq i < i' \leq 6$, $E_i + (-1)^{i+i'} E_{i'}$ gives one of a subset of 15 usual Plücker equations of 4 terms.
		
		\item\label{it:set3}
		$\binom{n}{q, p-2-q,p+2-q, n+q-2p}$ equations of $\binom{p+2-q}{2} \geq 15$ 
			\SELF{$\leq \binom{n-p+2}{2}$}
		terms,
		for each integer $q \in [\max\{0,2p-n\} , p-4]$.
		\SELF{Implies $4 \leq p \leq n-4$}
	\end{enumerate}	
\end{proposition}
\begin{proof}
	We can write \eqref{eq:neoPlucker}, with
	$\qq = \jj \cap \kk$, 
	$q = |\qq| \in [\max\{0,2p-n\},p-2]$,
		\SELF{$\leq n-4$}
	$\jj' = \jj\backslash\kk = (j'_1,\ldots,j'_{p-2-q})$
	and
	$\kk' = \kk\backslash\jj = (k'_1,\ldots,k'_{p+2-q})$,
	as
	\begin{align}
		\sum_{1\leq a < b \leq p+2-q} (-1)^{|\jj' > k'_a k'_b| + |\kk' > k'_a k'_b|} \lambda_{\jj\cup (k'_a k'_b)} \lambda_{\kk \backslash (k'_a k'_b)} &= \nonumber\\
		\sum_{1\leq a < b \leq p+2-q} (-1)^{|k'_a < \jj' < k'_b| + a + b} \lambda_{\qq \cup \jj' \cup (k'_a k'_b)} \lambda_{\qq \cup \kk' \backslash (k'_a k'_b)}
		&=0, \label{eq:neoPlucker2}
	\end{align}
	where $|k'_a < \jj' < k'_b|$ is the number of indices $j' \in \jj'$ with $k'_a < j' < k'_b$,
	and we used 
	$(-1)^{|\jj' > k'_a k'_b|} = (-1)^{|\jj' > k'_a| + |\jj' > k'_b|} = (-1)^{|k'_a < \jj' < k'_b| + 2|\jj' > k'_b|}$
		and
	$(-1)^{|\kk' > k'_a k'_b|} = (-1)^{(p+2-q-a)+(p+2-q-b)} = (-1)^{a+b}$.
	The following characterization will show the equations are distinct and non-trivial.
	
	(\ref{it:set1})
	If $q=p-2$ then $\jj' = \emptyset$ and $\kk' = (k'_1,\ldots,k'_4)$, so \eqref{eq:neoPlucker2} becomes
	\begin{equation*}
		\sum_{1\leq a < b \leq 4} (-1)^{a+b} \lambda_{\qq\cup (k'_a k'_b)} \lambda_{\qq \cup (k'_c k'_d)} = 0,
	\end{equation*}
	where $(c,d) = (1,2,3,4)\backslash (a,b)$.
	Each term appears twice and with the same sign, as $(-1)^{c+d} = (-1)^{a+b}$.
	Simplifying we find a Plücker equation%
	\footnote{Obtained 4 times in \eqref{eq:Plucker}, for $(\jj,\kk) = (\qq \cup k'_a,\qq\cup\kk'\backslash k'_a)$ with $a=1,\ldots,4$.},
	\begin{equation}\label{eq:neoPlucker3}
		\lambda_{\qq \cup (k'_1 k'_2)} \lambda_{\qq \cup (k'_3 k'_4)} -
		\lambda_{\qq \cup (k'_1 k'_3)} \lambda_{\qq \cup (k'_2 k'_4)} +
		\lambda_{\qq \cup (k'_1 k'_4)} \lambda_{\qq \cup (k'_2 k'_3)} = 0.
	\end{equation}
	The choices of $p-2$ indices for $\qq$ and other $4$ distinct ones for $\kk'$ yield  $\binom{n}{4, p-2, n-p-2}$ such equations.
	They are all distinct, as in \eqref{eq:neoPlucker3} we can determine $\qq$ (indices present in all $\lambda$'s) and $\kk'$ (remaining ones).
	
	(\ref{it:set2})
	If $q = p-3$ then $\jj' = (j')$, $\kk' = (k'_1,\ldots,k'_5)$,	
	and \eqref{eq:neoPlucker2} has $\binom{5}{2} = 10$ distinct terms:%
	\footnote{Note that $|k'_a < j' < k'_b|$ is 1 if the inequality holds, 0 otherwise.}
	\begin{equation}\label{eq:neoPLucker10}
		\sum_{1\leq a < b \leq 5} (-1)^{|k'_a < j' < k'_b| + a+b} \lambda_{\qq\cup (j'k'_a k'_b)} \lambda_{\qq \cup \kk' \backslash (k'_a k'_b)} = 0.
	\end{equation}
	Each of the $\binom{n}{6, p-3, n-p-3}$ choices for $\qq$ and
	$\ll = \jj' \cup \kk' = (l_1,\ldots,l_6)$ gives a set of 6 equations $\{E_1,\ldots,E_6\}$ with the same terms,
	where $E_i$ has $j' = l_i$ and $\kk' = \ll \backslash l_i$ (and so $k'_a = l_c$ with $a = c - |c>i|$).
	The sets are disjoint, as \eqref{eq:neoPLucker10} determines $\qq$ (indices in all $\lambda$'s) and $\ll$ (remaining ones). 
	
	We can write $E_i$ as follows, where $l_{icd} = (l_i, l_c, l_d)$ for short:
	\begin{align*}
		\sum_{1\leq a < b \leq 5} (-1)^{|k'_a < l_i < k'_b| + a+b} \lambda_{\qq\cup (l_i k'_a k'_b)} \lambda_{\qq \cup \kk' \backslash (k'_a k'_b)} & = \\
		\sum_{\substack{1\leq c < d \leq 6 \\ c,d \neq i}} (-1)^{|c < i < d| + c-|c>i|+d-|d>i|} \lambda_{\qq\cup l_{icd}} \lambda_{\qq \cup \ll \backslash l_{icd}} & = \\
		\sum_{\substack{1\leq c < d \leq 6 \\ c,d \neq i}} (-1)^{c+d} \lambda_{\qq\cup l_{icd}} \lambda_{\qq \cup \ll \backslash l_{icd}} & = 0.
	\end{align*}
	Given $1 \leq i'  \leq 6$, $i' \neq i$, we can split this sum and rewrite $E_i$ as
	\begin{equation*}
		\sum_{\substack{1\leq c < d \leq 6 \\ c,d \neq i,i'}} (-1)^{c+d} \lambda_{\qq\cup l_{icd}} \lambda_{\qq \cup \ll \backslash l_{icd}} 
		+
		\sum_{\substack{1\leq c \leq 6 \\ c \neq i,i'}} (-1)^{c+i'} \lambda_{\qq\cup l_{ii'c}} \lambda_{\qq \cup \ll \backslash l_{ii'c}}
		= 0.
	\end{equation*}
	As $\ll \backslash l_{icd} = l_{i'c'd'}$ for
	$(c',d') = (1,2,3,4,5,6)\backslash (i,i',c,d)$,
	and $(-1)^{c+d} = (-1)^{1+\cdots+6 - i - i' - c' - d'} = (-1)^{1 + i + i' + c' + d'}$,
	taking $E_i + (-1)^{i+i'} E_{i'}$
	the first sum cancels out, leaving a Plücker equation%
		\footnote{Obtained in \eqref{eq:Plucker} for $(\jj,\kk) = (\qq \cup l_{ii'},\qq \cup \ll\backslash l_{ii'})$.}
	of 4 terms,
		\OMIT{after factoring out $2\cdot (-1)^{i'}$}
	\begin{equation}\label{eq:old4}
		\sum_{\substack{1\leq c \leq 6 \\ c \neq i,i'}} (-1)^{c} \lambda_{\qq\cup l_{ii'c}} \lambda_{\qq \cup \ll \backslash l_{ii'c}}
		= 0.
	\end{equation}
	Each $i$ and $i'$ gives one of $\binom{6}{2} = 15$ different Plücker equations, 
	since $l_i$ and $l_{i'}$ can be determined as the indices always appearing together in only one $\lambda$ of each term.
	Also, $E_i \neq E_{i'}$, and so \eqref{eq:neoPlucker} has $6 \cdot \binom{n}{6, p-3, n-p-3} = \binom{n}{1,5, p-3, n-p-3}$ distinct equations of 10 terms.

	(\ref{it:set3})
	For $\max\{0,2p-n\} \leq q \leq p-4$, the number of distinct terms in \eqref{eq:neoPlucker2} is $\binom{p+2-q}{2} \geq \binom{p+2-(p-4)}{2} = 15$, so each $q$ gives a completely different set of equations.
		\OMIT{$\binom{p+2-q}{2}$ decreases with $q$}
	For a given $q$, each of the $\binom{n}{q, p-2-q,p+2-q, n+q-2p}$ choices for $\qq$, $\jj'$ and $\kk'$ yields a distinct equation, as in \eqref{eq:neoPlucker2} we can identify $\qq$ (indices in all $\lambda$'s), $\jj'$ ($p-2-q \geq 2$ remaining indices which in all terms are in the same $\lambda$) and $\kk'$ (the rest).
\end{proof}

\section*{Acknowledgments}

We thank Enzo Passos Silva Ribeiro, a Mathematics student at Universidade Federal da Bahia, for helping us improve the Python code.

%


\providecommand{\bysame}{\leavevmode\hbox to3em{\hrulefill}\thinspace}
\providecommand{\MR}{\relax\ifhmode\unskip\space\fi MR }
\providecommand{\MRhref}[2]{%
	\href{http://www.ams.org/mathscinet-getitem?mr=#1}{#2}
}
\providecommand{\href}[2]{#2}

\appendix

\section{Code and tables}\label{sc:Code and tables}

\begin{adjustwidth}{-7mm}{-7mm}
\begin{lstlisting}[language = {Python}, frame = single, linewidth=11.8cm, caption = {Python code to generate Plücker or Plücker-like equations}, label = {code}]
from IPython.display import display,Math
from itertools import combinations
from sympy import init_printing,IndexedBase,latex
init_printing()

n = 6  # Dimension of the space (3 < n < 10).
p = 3  # Dimension of the subspaces (1 < p < n-1).
m = 2  # m = 1 for the usual Plucker equations, 
       # m = 2 for the new Plucker-like equations.

lamb = IndexedBase('\\lambda')
indices = list(range(1,n+1))
J_values = list(combinations(indices,p-m))
K_values = list(combinations(indices,p+m))
for J in J_values:
  for K in K_values:
    expr = 0
    K_minus_J = set(K) - set(J)
    J_symdif_K = set(J).symmetric_difference(K)
    for I in combinations(K_minus_J, m):
      sign = (-1)**sum(l > i for l in J_symdif_K for i in I)
      multiindex1 = ''.join(map(str,sorted(set(J).union(I))))
      multiindex2 = ''.join(map(str,sorted(set(K) - set(I))))
      expr += sign*lamb[multiindex1]*lamb[multiindex2]
    J_str = ''.join(map(str,J))
    K_str = ''.join(map(str,K))
    eq = latex(expr) + '=0'
    display(Math(f'({J_str},{K_str}):\\ {eq}'))	
\end{lstlisting}
\end{adjustwidth}

\newpage

\begin{longtable}{rll}
	\caption{Plücker equations for (n,p) = (6,3)} \label{tb:old} \\ 
	\# & $(\jj,\kk)$ & Equation \\
	\hline
	1 & (12,1234) & $0 = 0$ \\
	2 & (12,1235) & $0 = 0$ \\
	3 & (12,1236) & $0 = 0$ \\
	4 & (12,1245) & $0 = 0$ \\
	5 & (12,1246) & $0 = 0$ \\
	6 & (12,1256) & $0 = 0$ \\
	7 & (12,1345) & ${\lambda}_{123} {\lambda}_{145} - {\lambda}_{124} {\lambda}_{135} + {\lambda}_{125} {\lambda}_{134} = 0$ \\
	8 & (12,1346) & ${\lambda}_{123} {\lambda}_{146} - {\lambda}_{124} {\lambda}_{136} + {\lambda}_{126} {\lambda}_{134} = 0$ \\
	9 & (12,1356) & ${\lambda}_{123} {\lambda}_{156} - {\lambda}_{125} {\lambda}_{136} + {\lambda}_{126} {\lambda}_{135} = 0$ \\
	10 & (12,1456) & ${\lambda}_{124} {\lambda}_{156} - {\lambda}_{125} {\lambda}_{146} + {\lambda}_{126} {\lambda}_{145} = 0$ \\
	11 & (12,2345) & ${\lambda}_{123} {\lambda}_{245} - {\lambda}_{124} {\lambda}_{235} + {\lambda}_{125} {\lambda}_{234} = 0$ \\
	12 & (12,2346) & ${\lambda}_{123} {\lambda}_{246} - {\lambda}_{124} {\lambda}_{236} + {\lambda}_{126} {\lambda}_{234} = 0$ \\
	13 & (12,2356) & ${\lambda}_{123} {\lambda}_{256} - {\lambda}_{125} {\lambda}_{236} + {\lambda}_{126} {\lambda}_{235} = 0$ \\
	14 & (12,2456) & ${\lambda}_{124} {\lambda}_{256} - {\lambda}_{125} {\lambda}_{246} + {\lambda}_{126} {\lambda}_{245} = 0$ \\
	15 & (12,3456) & ${\lambda}_{123} {\lambda}_{456} - {\lambda}_{124} {\lambda}_{356} + {\lambda}_{125} {\lambda}_{346} - {\lambda}_{126} {\lambda}_{345} = 0$ \\
	16 & (13,1234) & $0 = 0$ \\
	17 & (13,1235) & $0 = 0$ \\
	18 & (13,1236) & $0 = 0$ \\
	19 & (13,1245) & ${\lambda}_{123} {\lambda}_{145} - {\lambda}_{124} {\lambda}_{135} + {\lambda}_{125} {\lambda}_{134} = 0$ \\
	\vdots \\
	29 & (13,2456) & ${\lambda}_{123} {\lambda}_{456} + {\lambda}_{134} {\lambda}_{256} - {\lambda}_{135} {\lambda}_{246} + {\lambda}_{136} {\lambda}_{245} = 0$ \\
	\vdots \\
	43 & (14,2356) & ${\lambda}_{124} {\lambda}_{356} - {\lambda}_{134} {\lambda}_{256} - {\lambda}_{145} {\lambda}_{236} + {\lambda}_{146} {\lambda}_{235} = 0$ \\
	\vdots \\
	57 & (15,2346) & ${\lambda}_{125} {\lambda}_{346} - {\lambda}_{135} {\lambda}_{246} + {\lambda}_{145} {\lambda}_{236} + {\lambda}_{156} {\lambda}_{234} = 0$ \\
	\vdots \\
	71 & (16,2345) & ${\lambda}_{126} {\lambda}_{345} - {\lambda}_{136} {\lambda}_{245} + {\lambda}_{146} {\lambda}_{235} - {\lambda}_{156} {\lambda}_{234} = 0$ \\
	\vdots \\
	211 & (56,1234) & ${\lambda}_{123} {\lambda}_{456} - {\lambda}_{124} {\lambda}_{356} + {\lambda}_{134} {\lambda}_{256} - {\lambda}_{156} {\lambda}_{234} = 0$ \\
	212 & (56,1235) & ${\lambda}_{125} {\lambda}_{356} - {\lambda}_{135} {\lambda}_{256} + {\lambda}_{156} {\lambda}_{235} = 0$ \\
	213 & (56,1236) & ${\lambda}_{126} {\lambda}_{356} - {\lambda}_{136} {\lambda}_{256} + {\lambda}_{156} {\lambda}_{236} = 0$ \\
	214 & (56,1245) & ${\lambda}_{125} {\lambda}_{456} - {\lambda}_{145} {\lambda}_{256} + {\lambda}_{156} {\lambda}_{245} = 0$ \\
	215 & (56,1246) & ${\lambda}_{126} {\lambda}_{456} - {\lambda}_{146} {\lambda}_{256} + {\lambda}_{156} {\lambda}_{246} = 0$ \\
	216 & (56,1256) & $0 = 0$ \\
	217 & (56,1345) & ${\lambda}_{135} {\lambda}_{456} - {\lambda}_{145} {\lambda}_{356} + {\lambda}_{156} {\lambda}_{345} = 0$ \\
	218 & (56,1346) & ${\lambda}_{136} {\lambda}_{456} - {\lambda}_{146} {\lambda}_{356} + {\lambda}_{156} {\lambda}_{346} = 0$ \\
	219 & (56,1356) & $0 = 0$ \\
	220 & (56,1456) & $0 = 0$ \\
	221 & (56,2345) & ${\lambda}_{235} {\lambda}_{456} - {\lambda}_{245} {\lambda}_{356} + {\lambda}_{256} {\lambda}_{345} = 0$ \\
	222 & (56,2346) & ${\lambda}_{236} {\lambda}_{456} - {\lambda}_{246} {\lambda}_{356} + {\lambda}_{256} {\lambda}_{346} = 0$ \\
	223 & (56,2356) & $0 = 0$ \\
	224 & (56,2456) & $0 = 0$ \\
	225 & (56,3456) & $0 = 0$ \\
\end{longtable}

\newpage

\begin{longtable}{rl}
	\caption{Plücker eqs.\ for (n,p) = (6,3), minus trivial or repeated ones} \label{tb:old_reduced} \\ 
	\# & Equation \\
	\hline
	1 & ${\lambda}_{123} {\lambda}_{145} - {\lambda}_{124} {\lambda}_{135} + {\lambda}_{125} {\lambda}_{134} = 0$ \\
	2 & ${\lambda}_{123} {\lambda}_{146} - {\lambda}_{124} {\lambda}_{136} + {\lambda}_{126} {\lambda}_{134} = 0$ \\
	3 & ${\lambda}_{123} {\lambda}_{156} - {\lambda}_{125} {\lambda}_{136} + {\lambda}_{126} {\lambda}_{135} = 0$ \\
	4 & ${\lambda}_{123} {\lambda}_{245} - {\lambda}_{124} {\lambda}_{235} + {\lambda}_{125} {\lambda}_{234} = 0$ \\
	5 & ${\lambda}_{123} {\lambda}_{246} - {\lambda}_{124} {\lambda}_{236} + {\lambda}_{126} {\lambda}_{234} = 0$ \\
	6 & ${\lambda}_{123} {\lambda}_{256} - {\lambda}_{125} {\lambda}_{236} + {\lambda}_{126} {\lambda}_{235} = 0$ \\
	7 & ${\lambda}_{123} {\lambda}_{345} - {\lambda}_{134} {\lambda}_{235} + {\lambda}_{135} {\lambda}_{234} = 0$ \\
	8 & ${\lambda}_{123} {\lambda}_{346} - {\lambda}_{134} {\lambda}_{236} + {\lambda}_{136} {\lambda}_{234} = 0$ \\
	9 & ${\lambda}_{123} {\lambda}_{356} - {\lambda}_{135} {\lambda}_{236} + {\lambda}_{136} {\lambda}_{235} = 0$ \\
	10 & ${\lambda}_{123} {\lambda}_{456} + {\lambda}_{125} {\lambda}_{346} - {\lambda}_{135} {\lambda}_{246} + {\lambda}_{146} {\lambda}_{235} = 0$ \\
	11 & ${\lambda}_{123} {\lambda}_{456} + {\lambda}_{134} {\lambda}_{256} - {\lambda}_{135} {\lambda}_{246} + {\lambda}_{136} {\lambda}_{245} = 0$ \\
	12 & ${\lambda}_{123} {\lambda}_{456} - {\lambda}_{124} {\lambda}_{356} + {\lambda}_{125} {\lambda}_{346} - {\lambda}_{126} {\lambda}_{345} = 0$ \\
	13 & ${\lambda}_{123} {\lambda}_{456} - {\lambda}_{124} {\lambda}_{356} + {\lambda}_{134} {\lambda}_{256} - {\lambda}_{156} {\lambda}_{234} = 0$ \\
	14 & ${\lambda}_{123} {\lambda}_{456} - {\lambda}_{126} {\lambda}_{345} + {\lambda}_{136} {\lambda}_{245} - {\lambda}_{145} {\lambda}_{236} = 0$ \\
	15 & ${\lambda}_{123} {\lambda}_{456} - {\lambda}_{145} {\lambda}_{236} + {\lambda}_{146} {\lambda}_{235} - {\lambda}_{156} {\lambda}_{234} = 0$ \\
	16 & ${\lambda}_{124} {\lambda}_{156} - {\lambda}_{125} {\lambda}_{146} + {\lambda}_{126} {\lambda}_{145} = 0$ \\
	17 & ${\lambda}_{124} {\lambda}_{256} - {\lambda}_{125} {\lambda}_{246} + {\lambda}_{126} {\lambda}_{245} = 0$ \\
	18 & ${\lambda}_{124} {\lambda}_{345} - {\lambda}_{134} {\lambda}_{245} + {\lambda}_{145} {\lambda}_{234} = 0$ \\
	19 & ${\lambda}_{124} {\lambda}_{346} - {\lambda}_{134} {\lambda}_{246} + {\lambda}_{146} {\lambda}_{234} = 0$ \\
	20 & ${\lambda}_{124} {\lambda}_{356} + {\lambda}_{126} {\lambda}_{345} - {\lambda}_{135} {\lambda}_{246} + {\lambda}_{146} {\lambda}_{235} = 0$ \\
	21 & ${\lambda}_{124} {\lambda}_{356} - {\lambda}_{125} {\lambda}_{346} + {\lambda}_{136} {\lambda}_{245} - {\lambda}_{145} {\lambda}_{236} = 0$ \\
	22 & ${\lambda}_{124} {\lambda}_{356} - {\lambda}_{134} {\lambda}_{256} - {\lambda}_{145} {\lambda}_{236} + {\lambda}_{146} {\lambda}_{235} = 0$ \\
	23 & ${\lambda}_{124} {\lambda}_{356} - {\lambda}_{135} {\lambda}_{246} + {\lambda}_{136} {\lambda}_{245} + {\lambda}_{156} {\lambda}_{234} = 0$ \\
	24 & ${\lambda}_{124} {\lambda}_{456} - {\lambda}_{145} {\lambda}_{246} + {\lambda}_{146} {\lambda}_{245} = 0$ \\
	25 & ${\lambda}_{125} {\lambda}_{345} - {\lambda}_{135} {\lambda}_{245} + {\lambda}_{145} {\lambda}_{235} = 0$ \\
	26 & ${\lambda}_{125} {\lambda}_{346} - {\lambda}_{126} {\lambda}_{345} - {\lambda}_{134} {\lambda}_{256} + {\lambda}_{156} {\lambda}_{234} = 0$ \\
	27 & ${\lambda}_{125} {\lambda}_{346} - {\lambda}_{134} {\lambda}_{256} - {\lambda}_{136} {\lambda}_{245} + {\lambda}_{146} {\lambda}_{235} = 0$ \\
	28 & ${\lambda}_{125} {\lambda}_{346} - {\lambda}_{135} {\lambda}_{246} + {\lambda}_{145} {\lambda}_{236} + {\lambda}_{156} {\lambda}_{234} = 0$ \\
	29 & ${\lambda}_{125} {\lambda}_{356} - {\lambda}_{135} {\lambda}_{256} + {\lambda}_{156} {\lambda}_{235} = 0$ \\
	30 & ${\lambda}_{125} {\lambda}_{456} - {\lambda}_{145} {\lambda}_{256} + {\lambda}_{156} {\lambda}_{245} = 0$ \\
	31 & ${\lambda}_{126} {\lambda}_{345} + {\lambda}_{134} {\lambda}_{256} - {\lambda}_{135} {\lambda}_{246} + {\lambda}_{145} {\lambda}_{236} = 0$ \\
	32 & ${\lambda}_{126} {\lambda}_{345} - {\lambda}_{136} {\lambda}_{245} + {\lambda}_{146} {\lambda}_{235} - {\lambda}_{156} {\lambda}_{234} = 0$ \\
	33 & ${\lambda}_{126} {\lambda}_{346} - {\lambda}_{136} {\lambda}_{246} + {\lambda}_{146} {\lambda}_{236} = 0$ \\
	34 & ${\lambda}_{126} {\lambda}_{356} - {\lambda}_{136} {\lambda}_{256} + {\lambda}_{156} {\lambda}_{236} = 0$ \\
	35 & ${\lambda}_{126} {\lambda}_{456} - {\lambda}_{146} {\lambda}_{256} + {\lambda}_{156} {\lambda}_{246} = 0$ \\
	36 & ${\lambda}_{134} {\lambda}_{156} - {\lambda}_{135} {\lambda}_{146} + {\lambda}_{136} {\lambda}_{145} = 0$ \\
	37 & ${\lambda}_{134} {\lambda}_{356} - {\lambda}_{135} {\lambda}_{346} + {\lambda}_{136} {\lambda}_{345} = 0$ \\
	38 & ${\lambda}_{134} {\lambda}_{456} - {\lambda}_{145} {\lambda}_{346} + {\lambda}_{146} {\lambda}_{345} = 0$ \\
	39 & ${\lambda}_{135} {\lambda}_{456} - {\lambda}_{145} {\lambda}_{356} + {\lambda}_{156} {\lambda}_{345} = 0$ \\
	40 & ${\lambda}_{136} {\lambda}_{456} - {\lambda}_{146} {\lambda}_{356} + {\lambda}_{156} {\lambda}_{346} = 0$ \\
	41 & ${\lambda}_{234} {\lambda}_{256} - {\lambda}_{235} {\lambda}_{246} + {\lambda}_{236} {\lambda}_{245} = 0$ \\
	42 & ${\lambda}_{234} {\lambda}_{356} - {\lambda}_{235} {\lambda}_{346} + {\lambda}_{236} {\lambda}_{345} = 0$ \\
	43 & ${\lambda}_{234} {\lambda}_{456} - {\lambda}_{245} {\lambda}_{346} + {\lambda}_{246} {\lambda}_{345} = 0$ \\
	44 & ${\lambda}_{235} {\lambda}_{456} - {\lambda}_{245} {\lambda}_{356} + {\lambda}_{256} {\lambda}_{345} = 0$ \\
	45 & ${\lambda}_{236} {\lambda}_{456} - {\lambda}_{246} {\lambda}_{356} + {\lambda}_{256} {\lambda}_{346} = 0$ \\
\end{longtable} 

\newpage

\begin{longtable}{rlp{8cm}}
	\caption{New Plücker-like equations for (n,p) = (6,3)} \label{tb:new} \\ 
	\# & $(\jj,\kk)$ & Equation \\
	\hline
	1 & (1,12345) & ${\lambda}_{123} {\lambda}_{145} - {\lambda}_{124} {\lambda}_{135} + {\lambda}_{125} {\lambda}_{134} = 0$ \\
	2 & (1,12346) & ${\lambda}_{123} {\lambda}_{146} - {\lambda}_{124} {\lambda}_{136} + {\lambda}_{126} {\lambda}_{134} = 0$ \\
	3 & (1,12356) & ${\lambda}_{123} {\lambda}_{156} - {\lambda}_{125} {\lambda}_{136} + {\lambda}_{126} {\lambda}_{135} = 0$ \\
	4 & (1,12456) & ${\lambda}_{124} {\lambda}_{156} - {\lambda}_{125} {\lambda}_{146} + {\lambda}_{126} {\lambda}_{145} = 0$ \\
	5 & (1,13456) & ${\lambda}_{134} {\lambda}_{156} - {\lambda}_{135} {\lambda}_{146} + {\lambda}_{136} {\lambda}_{145} = 0$ \\
	6 & (1,23456) & ${\lambda}_{123} {\lambda}_{456} - {\lambda}_{124} {\lambda}_{356} + {\lambda}_{125} {\lambda}_{346} - {\lambda}_{126} {\lambda}_{345} + {\lambda}_{134} {\lambda}_{256} - {\lambda}_{135} {\lambda}_{246} + {\lambda}_{136} {\lambda}_{245} + {\lambda}_{145} {\lambda}_{236} - {\lambda}_{146} {\lambda}_{235} + {\lambda}_{156} {\lambda}_{234} = 0$ \\
	7 & (2,12345) & ${\lambda}_{123} {\lambda}_{245} - {\lambda}_{124} {\lambda}_{235} + {\lambda}_{125} {\lambda}_{234} = 0$ \\
	8 & (2,12346) & ${\lambda}_{123} {\lambda}_{246} - {\lambda}_{124} {\lambda}_{236} + {\lambda}_{126} {\lambda}_{234} = 0$ \\
	9 & (2,12356) & ${\lambda}_{123} {\lambda}_{256} - {\lambda}_{125} {\lambda}_{236} + {\lambda}_{126} {\lambda}_{235} = 0$ \\
	10 & (2,12456) & ${\lambda}_{124} {\lambda}_{256} - {\lambda}_{125} {\lambda}_{246} + {\lambda}_{126} {\lambda}_{245} = 0$ \\
	11 & (2,13456) & ${\lambda}_{123} {\lambda}_{456} - {\lambda}_{124} {\lambda}_{356} + {\lambda}_{125} {\lambda}_{346} - {\lambda}_{126} {\lambda}_{345} - {\lambda}_{134} {\lambda}_{256} + {\lambda}_{135} {\lambda}_{246} - {\lambda}_{136} {\lambda}_{245} - {\lambda}_{145} {\lambda}_{236} + {\lambda}_{146} {\lambda}_{235} - {\lambda}_{156} {\lambda}_{234} = 0$ \\
	12 & (2,23456) & ${\lambda}_{234} {\lambda}_{256} - {\lambda}_{235} {\lambda}_{246} + {\lambda}_{236} {\lambda}_{245} = 0$ \\
	13 & (3,12345) & ${\lambda}_{123} {\lambda}_{345} - {\lambda}_{134} {\lambda}_{235} + {\lambda}_{135} {\lambda}_{234} = 0$ \\
	14 & (3,12346) & ${\lambda}_{123} {\lambda}_{346} - {\lambda}_{134} {\lambda}_{236} + {\lambda}_{136} {\lambda}_{234} = 0$ \\
	15 & (3,12356) & ${\lambda}_{123} {\lambda}_{356} - {\lambda}_{135} {\lambda}_{236} + {\lambda}_{136} {\lambda}_{235} = 0$ \\
	16 & (3,12456) & ${\lambda}_{123} {\lambda}_{456} + {\lambda}_{124} {\lambda}_{356} - {\lambda}_{125} {\lambda}_{346} + {\lambda}_{126} {\lambda}_{345} + {\lambda}_{134} {\lambda}_{256} - {\lambda}_{135} {\lambda}_{246} + {\lambda}_{136} {\lambda}_{245} - {\lambda}_{145} {\lambda}_{236} + {\lambda}_{146} {\lambda}_{235} - {\lambda}_{156} {\lambda}_{234} = 0$ \\
	17 & (3,13456) & ${\lambda}_{134} {\lambda}_{356} - {\lambda}_{135} {\lambda}_{346} + {\lambda}_{136} {\lambda}_{345} = 0$ \\
	18 & (3,23456) & ${\lambda}_{234} {\lambda}_{356} - {\lambda}_{235} {\lambda}_{346} + {\lambda}_{236} {\lambda}_{345} = 0$ \\
	19 & (4,12345) & ${\lambda}_{124} {\lambda}_{345} - {\lambda}_{134} {\lambda}_{245} + {\lambda}_{145} {\lambda}_{234} = 0$ \\
	20 & (4,12346) & ${\lambda}_{124} {\lambda}_{346} - {\lambda}_{134} {\lambda}_{246} + {\lambda}_{146} {\lambda}_{234} = 0$ \\
	21 & (4,12356) & ${\lambda}_{123} {\lambda}_{456} + {\lambda}_{124} {\lambda}_{356} + {\lambda}_{125} {\lambda}_{346} - {\lambda}_{126} {\lambda}_{345} - {\lambda}_{134} {\lambda}_{256} - {\lambda}_{135} {\lambda}_{246} + {\lambda}_{136} {\lambda}_{245} - {\lambda}_{145} {\lambda}_{236} + {\lambda}_{146} {\lambda}_{235} + {\lambda}_{156} {\lambda}_{234} = 0$ \\
	22 & (4,12456) & ${\lambda}_{124} {\lambda}_{456} - {\lambda}_{145} {\lambda}_{246} + {\lambda}_{146} {\lambda}_{245} = 0$ \\
	23 & (4,13456) & ${\lambda}_{134} {\lambda}_{456} - {\lambda}_{145} {\lambda}_{346} + {\lambda}_{146} {\lambda}_{345} = 0$ \\
	24 & (4,23456) & ${\lambda}_{234} {\lambda}_{456} - {\lambda}_{245} {\lambda}_{346} + {\lambda}_{246} {\lambda}_{345} = 0$ \\
	25 & (5,12345) & ${\lambda}_{125} {\lambda}_{345} - {\lambda}_{135} {\lambda}_{245} + {\lambda}_{145} {\lambda}_{235} = 0$ \\
	26 & (5,12346) & ${\lambda}_{123} {\lambda}_{456} - {\lambda}_{124} {\lambda}_{356} - {\lambda}_{125} {\lambda}_{346} - {\lambda}_{126} {\lambda}_{345} + {\lambda}_{134} {\lambda}_{256} + {\lambda}_{135} {\lambda}_{246} + {\lambda}_{136} {\lambda}_{245} - {\lambda}_{145} {\lambda}_{236} - {\lambda}_{146} {\lambda}_{235} - {\lambda}_{156} {\lambda}_{234} = 0$ \\
	27 & (5,12356) & ${\lambda}_{125} {\lambda}_{356} - {\lambda}_{135} {\lambda}_{256} + {\lambda}_{156} {\lambda}_{235} = 0$ \\
	28 & (5,12456) & ${\lambda}_{125} {\lambda}_{456} - {\lambda}_{145} {\lambda}_{256} + {\lambda}_{156} {\lambda}_{245} = 0$ \\
	29 & (5,13456) & ${\lambda}_{135} {\lambda}_{456} - {\lambda}_{145} {\lambda}_{356} + {\lambda}_{156} {\lambda}_{345} = 0$ \\
	30 & (5,23456) & ${\lambda}_{235} {\lambda}_{456} - {\lambda}_{245} {\lambda}_{356} + {\lambda}_{256} {\lambda}_{345} = 0$ \\
	31 & (6,12345) & ${\lambda}_{123} {\lambda}_{456} - {\lambda}_{124} {\lambda}_{356} + {\lambda}_{125} {\lambda}_{346} + {\lambda}_{126} {\lambda}_{345} + {\lambda}_{134} {\lambda}_{256} - {\lambda}_{135} {\lambda}_{246} - {\lambda}_{136} {\lambda}_{245} + {\lambda}_{145} {\lambda}_{236} + {\lambda}_{146} {\lambda}_{235} - {\lambda}_{156} {\lambda}_{234} = 0$ \\
	32 & (6,12346) & ${\lambda}_{126} {\lambda}_{346} - {\lambda}_{136} {\lambda}_{246} + {\lambda}_{146} {\lambda}_{236} = 0$ \\
	33 & (6,12356) & ${\lambda}_{126} {\lambda}_{356} - {\lambda}_{136} {\lambda}_{256} + {\lambda}_{156} {\lambda}_{236} = 0$ \\
	34 & (6,12456) & ${\lambda}_{126} {\lambda}_{456} - {\lambda}_{146} {\lambda}_{256} + {\lambda}_{156} {\lambda}_{246} = 0$ \\
	35 & (6,13456) & ${\lambda}_{136} {\lambda}_{456} - {\lambda}_{146} {\lambda}_{356} + {\lambda}_{156} {\lambda}_{346} = 0$ \\
	36 & (6,23456) & ${\lambda}_{236} {\lambda}_{456} - {\lambda}_{246} {\lambda}_{356} + {\lambda}_{256} {\lambda}_{346} = 0$ \\
\end{longtable}

\end{document}